\documentclass[review,3p,numbers,square]{elsarticle}

\usepackage{amssymb,amsmath,amsthm}
\usepackage{fancyhdr}
\usepackage{hyperref}
\hypersetup{%
  pdfauthor={Lishun XIAO},%
  pdfstartview=FitH,%
  CJKbookmarks=true,%
  bookmarksnumbered=true,%
  bookmarksopen=true,%
  colorlinks=true, linkcolor=blue, urlcolor=blue, citecolor=blue, %
}

\newcommand{\rtn}{\mathrm{\mathbf{R}}}
\newcommand*{\PR}{\mathrm{\mathbf{P}}}
\newcommand*{\EX}{\mathrm{\mathbf{E}}}
\newcommand*{\dif}{\,\mathrm{d}}
\newcommand*{\calF}{\mathcal{F}}
\newcommand*{\calL}{\mathcal{L}}
\newcommand{\calS}{\mathcal{S}}

\newcommand{\calH}{\mathcal{H}}

\newcommand{\calU}{\mathcal{U}}

\newcommand*{\prs}{\PR\text{\,--\,\,}a.s.}
\newcommand*{\pts}{\dif\PR\!\times\!\!\dif t\text{\,--\,\,}a.e.}

\newcommand{\tT}{[0,T]}
\newcommand{\intT}[2][T]{\int^{#1}_{#2}}

\newcommand{\one}[1]{{\bf 1}_{#1}}
\newcommand{\vp}{\varepsilon}

\DeclareMathOperator*{\esssup}{esssup}

\makeatletter
\def\ps@pprintTitle{%
	\let\@oddhead\@empty
	\let\@evenhead\@empty
	\def\@oddfoot{\footnotesize\itshape		\hfill\today}%
	\let\@evenfoot\@oddfoot}
\makeatother

\usepackage{cleveref}
\crefname{thm}{Theorem}{Theorems}
\crefname{pro}{Proposition}{Propositions}
\crefname{lem}{Lemma}{Lemmas}
\crefname{rmk}{Remark}{Remarks}
\crefname{cor}{Corollary}{Corollaries}
\crefname{dfn}{Definition}{Definitions}
\crefname{ex}{Example}{Examples}
\crefname{section}{Section}{Sections}
\crefname{subsection}{Subsection}{Subsections}

\newtheorem{thm}{Theorem}

\newtheorem{rmk}[thm]{Remark}
\newtheorem{dfn}[thm]{Definition}

\journal{}
\allowdisplaybreaks[4]

\begin{document}
	
\begin{frontmatter}

\title{{\boldmath\bf Probabilistic interpretation of HJB equations by the representation theorem for generators of BSDEs}\tnoteref{found}}
\tnotetext[found]{Supported by the National Natural Science Foundation of China (Nos.\;11371362 and 11601509) and the Natural Science Foundation of Jiangsu Province (No.\;BK20150167).}

\author{Lishun Xiao\corref{cor1}}
\ead{xiaolishun@cumt.edu.cn}
\cortext[cor1]{Corresponding author}
\author{Shengjun Fan}%
\author{Dejian Tian}%


\address{School of Mathematics, China University of Mining and Technology, Xuzhou, Jiangsu, 221116, P.R. China}

\begin{abstract}
  The purpose of this note is to propose a new  approach for the probabilistic interpretation of Hamilton-Jacobi-Bellman equations associated with stochastic recursive optimal control problems,  utilizing the representation theorem for generators of backward stochastic differential equations. The key idea of our approach for proving this interpretation consists of transmitting the signs between the solution and generator via the identity given by representation theorem. Compared with existing methods, our approach seems to be more applicable for general settings. This can also be regarded as a new application of such representation theorem.
\end{abstract}

\end{frontmatter}

\section{Introduction}
\label{sec:Introduction}
After the pioneering work on nonlinear backward stochastic differential equations (BSDEs for abbreviation) by \citet*{PardouxPeng1990SCL}, the theory of nonlinear BSDEs has been applied to many fields (see \citet*{ElKarouiPengQuenez1997MF} for details). An important application of BSDEs lies in stochastic control problem. \citet*{Peng1992SSR} first interpreted the viscosity solution of a generalized Hamilton-Jacobi-Bellman equation as the value function of a stochastic recursive optimal control problem which is described by a forward-backward SDE (FBSDE for short). \citet*{Peng1997SuiJiFenXiXuanJiang} introduced a notion of backward semigroups to demonstrate a generalized dynamic programming principle (DPP for short) for stochastic recursive optimal control problem. He also provided a method of approximation of BSDEs' solutions to prove the probabilistic interpretation for HJB equations, i.e., the value function is a viscosity solution of the HJB equation. In this method several BSDEs with different generators and some estimates for solutions are applied to obtain the required variational inequality (see \eqref{eq:ViscositySubsoultionInequality} or \cref{dfn:ViscositySolutionsOfHJBEquations}). 

Since then, many researchers began to investigate stochastic recursive optimal control problem induced by FBSDE systems. \citet*{BuckdahnLi2008SICON} studied zero-sum two-player stochastic differential games via FBSDEs;
recently, \citet*{BuckdahnNie2016SICON} considered a stochastic exit time optimal control problem. All the previous works manifested the probabilistic interpretation for corresponding HJB equations with Cauchy problems or Dirichlet boundary conditions by Peng's approximation method with some necessary technical modifications. Under a non-Lipschitz setting, \citet*{PuZhang2015arXiv} proved the probabilistic interpretation for HJB equations by the approximation of viscosity solution sequence.  

Above all, to our best knowledge, the existing methods to process the probabilistic interpretation of HJB equations are all based on the approximation of (BSDEs' or PDEs') solutions. In this note, we would like to propose a new and unified approach to treat this probabilistic interpretation utilizing the representation theorem for generators of BSDEs (see \cref{thm:RepresentationTheoremForGenerators}). This representation theorem was established by \citet*{BriandCoquetHuMeminPeng2000ECP} and further extended by \citet*{Jiang2005SPA,Jiang2008AAP}. Essentially, the very crucial step of proving the viscosity solution is to claim a variational inequality (see \eqref{eq:ViscositySubsoultionInequality} or \cref{dfn:ViscositySolutionsOfHJBEquations}), the left hand side (without the sup) of which is actually a generator of a BSDE. 
The novelty of our approach is that the signs of the required generator and the BSDE's solution inherit directly form each other via an identity given by the representation theorem. Moreover, by our approach we can observe that the probabilistic interpretation can be boiled down to the representation problem for generators of a BSDE, provided the DPP holds. So compared with existing methods, the representation theorem approach is more applicable to general frameworks. And this can also be seen as a new application of such representation theorem.

The rest of this paper is organized as follows: \cref{sec:Preliminaries} gives all necessary notations and some elementary results about BSDEs; \cref{sec:ProbabilisticInterpretationHJBEquations} illustrates the probabilistic interpretation for HJB equations in viscosity sense adopting the representation theorem for generators of BSDEs.

\section{Preliminaries}\label{sec:Preliminaries}
Let $T>0$ be a given finite time horizon, $(\Omega,\calF,\PR)$ a probability space carrying a standard $d$-dimensional Brownian motion $(B_t)_{t\geq 0}$ and $(\calF_t)_{t\geq 0}$ the natural $\sigma$-algebra filtration generated by $(B_t)_{t\geq 0}$ with $\calF_0$ containing all $\PR$-null sets of $\calF$. Postulate that $\calF_T=\calF$ and $(\calF_t)_{t\geq 0}$ satisfies the usual conditions. Throughout this note we use $|\cdot|$ and $\langle \cdot,\cdot\rangle$ to denote the Euclidean norm and dotproduct, respectively. The Euclidean norm of a matrix $z\in\rtn^{n\times d}$ will be denoted by $|z|:=\sqrt{Tr(zz^*)}$, where and hereafter $z^*$ represents the transpose of $z$. 
We denote by $\calS^2(0,T;\rtn)$ (or $\calS^2$ for brevity) the set of real valued, $(\calF_t)$-adapted and continuous processes $(y_t)_{t\in\tT}$ such that $\EX[\sup_{t\in\tT}|y_t|^2]<\infty$. Let $\calH^2(0,T;\rtn^n)$ (or $\calH^2$ for brevity) denote the set of $\rtn^n$-valued and $(\calF_t)$-progressively measurable processes $(z_t)_{t\in\tT}$ satisfying that $\EX[\intT{0}|z_s|^2\dif s]<\infty$. 

Next we introduce some elementary results about BSDEs of the following type:
\begin{equation*}
  Y_t=\xi+\int^T_tg(\omega,s,Y_s,Z_s)\dif s-\int^T_t\langle Z_s,\dif B_s\rangle,\quad t\in\tT.
\end{equation*}
If we assume that the terminal data $\xi$ is $\calF_T$-measurable and $\EX|\xi|^2<\infty$, the generator $g:\Omega\times\tT\times\rtn\times\rtn^d\mapsto\rtn$ is $(\calF_t)$-progressively measurable and satisfies 
\begin{enumerate}
	\renewcommand{\theenumi}{(A\arabic{enumi})}
	\renewcommand{\labelenumi}{\theenumi}
	\item\label{A:GSquareIntegrable} $\{g(\omega,t,0,0)\}_{t\in\tT}\in\calH^2$;
	\item\label{A:GLipschitz} There exists a constant $K\geq 0$ such that $\pts$, for each $y$, $y'\in\rtn$ and $z$, $z'\in\rtn^d$,
	\[|g(\omega,t,y,z)-g(\omega,t,y',z')|\leq K(|y-y'|+|z-z'|),\]
\end{enumerate}
then by the result of \citet*{PardouxPeng1990SCL} the previous BSDE admits a unique solution $(Y_t,Z_t)_{t\in\tT}$ in $\calS^2\times\calH^2$. 

We now proceed to introduce the representation theorem for generators of BSDEs. Assume that the generator $g$ satisfies \ref{A:GSquareIntegrable} -- \ref{A:GLipschitz} and fix a triplet $(t,y,z)\in[0,T)\times\rtn\times\rtn^d$. Then, for each $\vp$ with $0<\vp\leq T-t$, the following BSDE admits a unique solution $(Y^\vp_s,Z^\vp_s)_{s\in[0,t+\vp]}$ in $\calS^2\times\calH^2$:
\begin{equation}\label{eq:BSDERepresentationThm}
  Y^\vp_s=y+\langle z,B_{t+\vp}-B_t\rangle+\int^{t+\vp}_sg(\omega,r,Y^{\vp}_r,Z^{\vp}_r)\dif r-\int^{t+\vp}_s\langle Z^{\vp}_r,\dif B_r\rangle,\quad s\in[0,t+\vp].
\end{equation}
\begin{thm}[Theorem 3.3 in \citet*{Jiang2005SPA} and Lemma 2.1 in \citet*{Jiang2008AAP}]\label{thm:RepresentationTheoremForGenerators}
	Assume that $g$ satisfies \ref{A:GSquareIntegrable} -- \ref{A:GLipschitz} and $1\leq p<2$. Then for each $(t,y,z)\in[0,T)\times\rtn\times\rtn^d$,
	\begin{equation}\label{eq:RepThmWithoutLebesgue}
	  L^p-\lim_{\vp\to0^+}\frac{1}{\vp}\EX\bigg[(Y^\vp_t-y)-\int^{t+\vp}_tg(\omega,r,y,z)\dif r\bigg|\calF_t\bigg]=0;
	\end{equation} 
	and for each $(y,z)\in\rtn\times\rtn^d$, the following equality:
	\begin{equation}\label{eq:RepThmWithExplicitGenerator}
	  g(\omega,t,y,z)=L^p-\lim_{\vp\to0^+}\frac{1}{\vp}\big(Y^\vp_t-y\big)
	\end{equation}
	holds for almost every $t\in[0,T)$, where $Y^\vp_t$ is the solution of BSDE \eqref{eq:BSDERepresentationThm}. Moreover, if $g(\omega,\cdot,y,z)$ is continuous, the latter equality holds for all $t\in[0,T)$.
\end{thm}

\section{Probabilistic interpretation for HJB equations}\label{sec:ProbabilisticInterpretationHJBEquations}
In this section we will show the probabilistic interpretation for a generalized HJB equations, in viscosity sense, which are associated with stochastic recursive optimal control problems. Before proving the probabilistic interpretation, we should give a DPP for a stochastic recursive optimal control problem of the cost functional described by a controlled FBSDE system. This DPP is a well-known result and can be obtained by corresponding results in \citet*{Peng1997SuiJiFenXiXuanJiang} and \citet*{PuZhang2015arXiv}, so we omit its proof.

The set $\calU$ of admissible control processes is defined by 
\[\calU:=\big\{(v_t)_{t\in\tT}|v(\cdot)\in\calH^2(0,T;\rtn^k) \text{ and takes values in a compact set }U\subset\rtn^k\big\}.\]
For a given admissible control $v(\cdot)\in\calU$, we consider the following FBSDE system:
\begin{equation}\label{eq:ControlSystem}
\begin{cases}
  \displaystyle X^{t,x;v}_s=x+\int^s_tb(r,X^{t,x;v}_r,v_r)\dif r+\int^s_t\sigma(r,X^{t,x;v}_r,v_r)\dif B_r,\\
  \displaystyle 
  Y^{t,x;v}_s=\Phi(X^{t,x;v}_T)+\int^T_s g(r,X^{t,x;v}_r,Y^{t,x;v}_r,Z^{t,x;v}_r,v_r)\dif r-\int^T_s\langle Z^{t,x;v}_r,\dif B_r\rangle,\quad s\in[t,T],
\end{cases}
\end{equation}
where $t\in\tT$ is the initial time, $x\in\rtn^n$ is the initial state, and mappings $b:\tT\times\rtn^n\times U\mapsto\rtn^n$, $\sigma:\tT\times\rtn^n\times U\mapsto\rtn^{n\times d}$, 	$g:\tT\times\rtn^n\times\rtn\times\rtn^d\times U\mapsto\rtn$, $\Phi:\rtn^n\mapsto\rtn$ satisfy the following conditions:
\begin{enumerate}
	\renewcommand{\theenumi}{(H\arabic{enumi})}
	\renewcommand{\labelenumi}{\theenumi}
	\item\label{H:BSigmaContinuousInT} For each $x\in\rtn^n$, $y\in\rtn$, $z\in\rtn^d$ and $v\in U$, $t\mapsto b(t,x,v)$, $\sigma(t,x,v)$, $g(t,x,y,z,v)$ are continuous;
	\item\label{H:BSigmaLipschitz} There exists a constant $K\geq 0$ such that for each $x$,  $x'\in\rtn^n$, and $v$, $v'\in U$, 
	\[|b(t,x,v)-b(t,x',v')|+|\sigma(t,x,v)-\sigma(t,x',v')|\leq K(|x-x'|+|v-v'|);\]
	\item\label{H:GPhiLipschitz} There exists a constant $K\geq 0$ such that for each $x$, $x'\in\rtn^n$, $y$, $y'\in\rtn$, $z$, $z'\in\rtn^d$, $v$, $v'\in U$,
	\[|g(t,x,y,z,v)-g(t,x',y',z',v')|+|\Phi(x)-\Phi(x')|\leq K(|x-x'|+|y-y'|+|z-z'|+|v-v'|).\]
\end{enumerate}
Obviously, under the above assumptions, for any $v(\cdot)\in\calU$ the control system \eqref{eq:ControlSystem} admits a unique solution $(X^{t,x;v}_s,Y^{t,x;v}_s,Z^{t,x;v}_s)_{s\in[t,T]}$ in $\calS^2(t,T;\rtn^n\times\rtn)\times\calH^2(t,T;\rtn^d)$.

We now define the associated cost functional,
\begin{equation*}
  J(t,x;v(\cdot)):=Y^{t,x;v}_s|_{s=t},\quad (t,x)\in\tT\times\rtn^n,\;v(\cdot)\in\calU,
\end{equation*}
and define the value function of the stochastic recursive optimal control problem,
\begin{equation}\label{eq:ValueFunction}
  u(t,x):=\esssup_{v(\cdot)\in\calU}J(t,x;v(\cdot)),\quad (t,x)\in\tT\times\rtn^n.
\end{equation}
Here by standard estimates for FBSDE \eqref{eq:ControlSystem} we know that $u(t,x)$ is well defined. Moreover, $u(t,x)$ is deterministic, continuous in $(t,x)$ and of at most linear growth with respect to $x$, see \citet*{Peng1997SuiJiFenXiXuanJiang} or \citet*{PuZhang2015arXiv} for a survey. To introduce the DPP, we need the notion of backward semigroups, which is original from \citet*{Peng1997SuiJiFenXiXuanJiang}. For each $(t,x)\in\tT\times\rtn^n$, $v(\cdot)\in\calU$, $0\leq\delta\leq T-t$ and a random variable $\xi\in L^2(\Omega,\calF_{t+\delta},\PR;\rtn)$, we denote $G^{t,x;v}_{t,t+\delta}[\xi]:=Y_t$, where $(Y_s,Z_s)_{s\in[t,t+\delta]}$ is the solution of the following BSDE:
\begin{equation*}
  Y_s=\xi+\int^{t+\delta}_sg(r,X^{t,x;v}_r,Y_r,Z_r,v_r)\dif r-\int^{t+\delta}_s\langle Z_r,\dif B_r\rangle,\quad s\in[t,t+\delta].
\end{equation*}
Then for the control system \eqref{eq:ControlSystem} we have that $G^{t,x;v}_{t,T}[\Phi(X^{t,x;v}_T)]=G^{t,x;v}_{t,t+\delta}[Y^{t,x;v}_{t+\delta}]$.

\begin{thm}[DPP]\label{thm:DPP}
	Assume that \ref{H:BSigmaContinuousInT} -- \ref{H:GPhiLipschitz} hold. Then the value function $u(t,x)$ enjoys the following dynamic programming principle, for each $0\leq\delta\leq T-t$,
	\begin{equation*}
	  u(t,x)=\sup_{v(\cdot)\in\calU}G^{t,x;v}_{t,t+\delta}[u(t+\delta,X^{t,x;v}_{t+\delta})].
	\end{equation*}
\end{thm}

Next we will relate the value function \eqref{eq:ValueFunction} with the following generalized HJB equation, which is a fully nonlinear second order PDE of parabolic type: 
\begin{equation}\label{eq:HJBEquation}
  \begin{cases}
    \partial_t u(t,x)+\sup_{v\in U}\left\{\calL^v_tu(t,x)+g(t,x,u(t,x),\sigma^*(t,x,v)\nabla u(t,x),v)\right\}=0,\\
    u(T,x)=\Phi(x),\\
  \end{cases}
\end{equation}
where $\calL^v_t$ is a family of second order  partial differential operators, 
\[\calL^v_t u=\frac{1}{2}Tr\{(\sigma\sigma^*)(t,x,v)D^2u\}+\langle b(t,x,v),\nabla u\rangle.\]
We would like to prove the value function $u(t,x)$ defined in \eqref{eq:ValueFunction} is a viscosity solution of HJB equation \eqref{eq:HJBEquation}. We first recall the notion of viscosity solution for \eqref{eq:HJBEquation}, which is adapted from \citet*{CrandallIshiiLions1992BAMS} and \citet*{Peng1997SuiJiFenXiXuanJiang}. 
\begin{dfn}\label{dfn:ViscositySolutionsOfHJBEquations}
	A function $u\in C(\tT\times\rtn^n;\rtn)$ is called a viscosity subsolution (resp. supersolution) of HJB equation \eqref{eq:HJBEquation}, if $u(T,x)\leq\Phi(x)$ (resp. $u(T,x)\geq\Phi(x)$) for all $x\in\rtn^n$, and for any $\varphi\in C^{1,2}_b(\tT\times\rtn^n;\rtn)$ such that whenever $(t,x)\in[0,T)\times\rtn^n$ is a local minimum (resp. maximum) point of $\varphi-u$, then
	\begin{equation*}
	  \partial_t\varphi(t,x)+\sup_{v\in U}\left\{\calL^v_t\varphi(t,x)+g(t,x,u(t,x),\sigma^*(t,x,v)\nabla\varphi(t,x),v)\right\}\geq \text{ (resp. }\leq \text{) } 0.
	\end{equation*}
	A function $u\in C(\tT\times\rtn^n;\rtn)$ is called a viscosity solution of \eqref{eq:HJBEquation} if it is both a viscosity subsolution and a viscosity supersolution.
\end{dfn}

\begin{thm}\label{thm:ProbabilisticInterpretationForHJBEquations}
	Let assumptions \ref{H:BSigmaContinuousInT} -- \ref{H:GPhiLipschitz} hold. Then the value function $u(t,x)$ defined by \eqref{eq:ValueFunction} is a viscosity solution of HJB equation \eqref{eq:HJBEquation}.
\end{thm}

\begin{proof}
	Note that $u(t,x)$ is continuous in $(t,x)$. We first prove that $u$ is a viscosity supersolution. Take any $\varphi\in C^{1,2}_b(\tT\times\rtn^n;\rtn)$, $(t,x)\in[0,T)\times\rtn^n$ such that $\varphi-u$ achieves the local maximum $0$ at $(t,x)$. Without loss of generality, we assume $u(t,x)=\varphi(t,x)$. Since $u(T,x)=\Phi(x)$ holds for all $x\in\rtn^n$, it reduces to prove that	
  \begin{equation}\label{eq:ViscositySupersoultionInequality}
    \partial_t\varphi(t,x)+\sup_{v\in U}\left\{\calL^v_t\varphi(t,x)+g(t,x,u(t,x),\sigma^*(t,x,v)\nabla\varphi(t,x),v)\right\}\leq 0.
  \end{equation}		
  It follows from \cref{thm:DPP} that for each $0<\delta\leq T-t$,
  \[\varphi(t,x)=u(t,x)=\sup_{v(\cdot)\in\calU}G^{t,x;v}_{t,t+\delta}[u(t+\delta,X^{t,x;v}_{t+\delta})].\]
  The fact that $\varphi\leq u$ and the monotonicity of backward semigroup $G$ (or the comparison theorem for solutions of BSDEs, Theorem 2.2 in \citet*{ElKarouiPengQuenez1997MF}) yield that
  \begin{equation}\label{eq:SemigroupMinusVarphiLeqZero}
  \sup_{v(\cdot)\in\calU}\left\{G^{t,x;v}_{t,t+\delta}[\varphi(t+\delta,X^{t,x;v}_{t+\delta})]-\varphi(t,x)\right\}\leq 0.
  \end{equation}
  For each $v(\cdot)\in\calU$, we set $Y^{v,\delta}_t:=G^{t,x;v}_{t,t+\delta}[\varphi(t+\delta,X^{t,x;v}_{t+\delta})]$, which is a solution of the following BSDE,
  \[Y^{v,\delta}_t=\varphi(t+\delta,X^{t,x;v}_{t+\delta})+\int^{t+\delta}_tg(r,X^{t,x;v}_r,Y^{v,\delta}_r,Z^{v,\delta}_r,v_r)\dif r-\int^{t+\delta}_t\langle Z^{v,\delta}_r,\dif B_r\rangle.\]
  It\^o's formula to $\varphi(r,X^{t,x;v}_r)$ at the time interval $[t,t+\delta]$ reads 
  \begin{align*}
  \varphi(t,x)={}&\varphi(t+\delta,X^{t,x;v}_{t+\delta})-\int^{t+\delta}_t\left[\partial_r\varphi(r,X^{t,x;v}_r)+\calL^v_r\varphi(r,X^{t,x;v}_r)\right]\dif r\\
  &-\int^{t+\delta}_t\langle \sigma^*(r,X^{t,x;v}_r,v_r)\nabla\varphi(r,X^{t,x;v}_r),\dif B_r\rangle.
  \end{align*}
  Hence, we deduce that, setting $\hat Y^{v,\delta}_\cdot:=Y^{v,\delta}_\cdot-\varphi(\cdot,X^{t,x;v}_\cdot)$, $\hat Z^{v,\delta}_\cdot:=Z^{v,\delta}_\cdot-\sigma^*(\cdot,X^{t,x;v}_\cdot,v_\cdot)\nabla\varphi(\cdot,X^{t,x;v}_\cdot)$,
  \begin{equation}
  \hat Y^{v,\delta}_t=\int^{t+\delta}_tF(r,X^{t,x;v}_r,\hat Y^{v,\delta}_r,\hat Z^{v,\delta}_r,v_r)\dif r-\int^{t+\delta}_t\langle \hat Z^{v,\delta}_r,\dif B_r\rangle,\label{eq:BSDEWithCalLPlusG}
  \end{equation}
  where for each $r\in[t,t+\delta]$, $x\in\rtn^n$, $y\in\rtn$, $z\in\rtn^d$ and $v\in U$,
  \begin{align*}
  F(r,x,y,z,v):=\partial_r\varphi(r,x)+\calL^v_r\varphi(r,x)+g(r,x,y+\varphi(r,x),z+\sigma^*(r,x,v)\nabla\varphi(r,x),v).
  \end{align*}
  Now by \eqref{eq:SemigroupMinusVarphiLeqZero} we have that for each $0<\delta\leq T-t$, $\sup_{v(\cdot)\in\calU}\hat Y^{v,\delta}_t\leq 0$. Hence, we deduce that $\hat Y^{v',\delta}_t\leq 0$ holds for each $v'\in U$. It is evident that \ref{A:GSquareIntegrable} and \ref{A:GLipschitz} hold true for $F(r,X^{t,x;v}_r,y,z,v_r)$ since $g$ satisfies \ref{H:BSigmaContinuousInT} and \ref{H:GPhiLipschitz}. Then \eqref{eq:RepThmWithExplicitGenerator} in \cref{thm:RepresentationTheoremForGenerators} implies that for each $v'\in U$,
  \begin{equation*}
    F(t,x,0,0,v')=\lim_{\delta\to0}\frac{\hat Y^{v',\delta}_t}{\delta}\leq 0.
  \end{equation*}
  Thereby, we know that \eqref{eq:ViscositySupersoultionInequality} holds by the definition of $F$.

  Finally, we prove that $u$ is a viscosity subsolution. Take any $\varphi\in C^{1,2}_b(\tT\times\rtn^n;\rtn)$ such that $\varphi-u$ achieves the local minimum $0$ at $(t,x)\in[0,T)\times\rtn^n$. We only need to prove that	
  \begin{equation}\label{eq:ViscositySubsoultionInequality}
    \partial_t\varphi(t,x)+\sup_{v\in U}\left\{\calL^v_t\varphi(t,x)+g(t,x,u(t,x),\sigma^*(t,x,v)\nabla\varphi(t,x),v)\right\}\geq 0.
  \end{equation}	
  Analogous to the previous arguments, we will get that
  \begin{equation}\label{eq:SemigroupMinusVarphiGeqZero}
    \sup_{v(\cdot)\in\calU}\left\{G^{t,x;v}_{t,t+\delta}[\varphi(t+\delta,X^{t,x;v}_{t+\delta})]-\varphi(t,x)\right\}\geq 0,
  \end{equation}
  and BSDE \eqref{eq:BSDEWithCalLPlusG} still holds. Now we know that $\sup_{v(\cdot)\in\calU}\hat Y^{v,\delta}_t\geq 0$ for each $0<\delta\leq T-t$. Thus, there exists a sequence $\{v^i(\cdot)\}_{i\geq 1}\subset\calU$ such that, $\prs$, $\sup_{v(\cdot)\in\calU}\hat Y^{v,\delta}_t=\sup_{i\geq 1}\hat Y^{v^i,\delta}_t$. For each $0<\delta\leq T-t$ and $\vp>0$, we define
  \[\widetilde \Gamma_i:=\bigg\{\sup_{v(\cdot)\in\calU}\hat Y^{v,\delta}_t\leq \hat Y^{v^i,\delta}_t+\delta\vp\bigg\}\in\calF_t,\quad i\geq 1.\]
  Then the events $\Gamma_1:=\widetilde \Gamma_1$, $\widetilde \Gamma_i/(\cup^{i-1}_{j=1}\widetilde\Gamma_j)\in\calF_t$, $i\geq 2$ are mutually disjoint and form a $(\Omega,\calF_t)$-partition. Obviously, we have that $v^\vp(\cdot):=\sum_{i\geq 1}\one{\Gamma_i}v^i(\cdot)\in\calU$. And from the uniqueness for solutions of BSDEs, we konw that $\prs$, $\hat Y^{v^\vp,\delta}_t=\sum_{i\geq 1}\one{\Gamma_i}\hat Y^{v^i,\delta}_t$. Hence, we conclude that, $\prs$,
  \begin{equation*}
    \hat Y^{v^\vp,\delta}_t=\sum_{i\geq 1}\one{\Gamma_i}\hat Y^{v^i,\delta}_t\geq\sup_{v(\cdot)\in\calU}\hat Y^{v,\delta}_t-\delta\vp\geq-\delta\vp.
  \end{equation*}
  We now suppose that \eqref{eq:ViscositySubsoultionInequality} does not hold. Then by the definition of $F$, there exists a $\vp_0>0$ such that $F(t,x,0,0,v)<-\vp_0$ for each $v\in U$. Since $F(\cdot,x,0,0,\cdot)$ is uniformly continuous, when $\delta$ is small enough we derive that $F(r,x,0,0,v)<-\vp_0/2$ for each $r\in[t,t+\delta]$ and $v\in U$. Applying \eqref{eq:RepThmWithoutLebesgue} in \cref{thm:RepresentationTheoremForGenerators} we have that for each $t\in\tT$,
  \begin{equation*}
    \lim_{\delta\to0}\frac{1}{\delta}
    \EX\bigg[\bigg|\hat Y^{v^\vp,\delta}_t-\int^{t+\delta}_tF(r,X^{t,x;v^\vp}_r,0,0,v^\vp_r)\dif r\bigg|\bigg]=0.
  \end{equation*}
  Then, for $\vp_0/4>0$, there exists a small enough $\delta>0$ such that
  \begin{equation*}
    \frac{1}{\delta}\EX[\hat Y^{v^\vp,\delta}_t]\leq \frac{1}{\delta}\EX\bigg[\int^{t+\delta}_tF(r,X^{t,x;v^\vp}_r,0,0,v^\vp_r)\dif r\bigg]+\frac{\vp_0}{4}.
  \end{equation*}
  Noticing that $\hat Y^{v^\vp,\delta}_t\geq-\delta\vp$ and the following two estimates, where $C\geq 0$ and $p\geq 1$ are two constants,
  \begin{gather*}
    |F(r,X^{t,x;v^\vp}_r,0,0,v^\vp_r)-F(r,x,0,0,v^\vp_r)|\leq C(1+|x|^2)(|X^{t,x;v^\vp}_r-x|+|X^{t,x;v^\vp}_r-x|^3),\\
    \EX\bigg[\sup_{r\in[t,t+\delta]}|X^{t,x;v^\vp}_r-x|^p\bigg]\leq C\delta^{p/2},
  \end{gather*}
  we can take $\delta\to0$ and then combine with $\prs$, $F(r,x,0,0,v^\vp_r)<-\vp_0/2$, obtaining that 
  \begin{align*}
  -\vp\leq 
  \lim_{\delta\to0}\frac{1}{\delta}\EX\bigg[\int^{t+\delta}_tF(r,x,0,0,v^\vp_r)\dif r\bigg]\leq \frac{\vp_0}{4}-\frac{\vp_0}{2}=-\frac{\vp_0}{4}.
  \end{align*}
  Hence, when $\vp=\vp_0/6$ it will contradict with $\vp_0>0$. So the inequality \eqref{eq:ViscositySubsoultionInequality} holds. Therefore, the proof of  \cref{thm:ProbabilisticInterpretationForHJBEquations} is finished. 
\end{proof}

\begin{rmk}
	(i) From the proof procedure of \cref{thm:ProbabilisticInterpretationForHJBEquations}, we can observe that if the DPP holds true, the probabilistic interpretation for HJB equations can be proved as soon as the representation theorem for generators of BSDE \eqref{eq:BSDEWithCalLPlusG} holds true, where such theorem is determined by the conditions for the generator $g$ of the BSDE in \eqref{eq:ControlSystem}. So our method also applies to the cases of \citet*{BuckdahnLi2008SICON} and \citet*{PuZhang2015arXiv}. Moreover, the additional assumption, adopted in \citet*{PuZhang2015arXiv}, that $g(t,x,y,z)$ is independent of $z$ can be eliminated naturally by the representation theorem in \citet*{FanJiangXu2011EJP}. 

  (ii) The representation theorem for generators of BSDEs can also be applied to prove the probabilistic interpretation for semilinear (or quasilinear) second order PDEs of both elliptic and parabolic types, just omit the control process $v(\cdot)$ in \eqref{eq:HJBEquation}. So the representation theorem method can be regarded as a unified approach to the probabilistic interpretation for semilinear, quasilinear and HJB type PDEs.
\end{rmk}

\section*{Acknowledgements}
The first author expresses his gratitude to Dr. Yu Zhuo (School of Mathematical Sciences, Fudan University) for her many helpful discussions that lead to the improved version of this note.

\end{document}